\documentclass[a4paper, 11pt]{amsart}

\usepackage[T1]{fontenc}
\usepackage[utf8]{inputenc}

\usepackage{amssymb}

\makeatletter
\@namedef{subjclassname@2010}{%
  \textup{2020} Mathematics Subject Classification}
\makeatother

\allowdisplaybreaks[2]

\newtheorem{theorem}{Theorem}[section]

\newtheorem{lemma}[theorem]{Lemma}
\newtheorem{corollary}[theorem]{Corollary}
\theoremstyle{definition}
\newtheorem{remark}[theorem]{Remark}

\newtheorem{example}[theorem]{Example}
\newtheorem{question}[theorem]{Question}
\numberwithin{equation}{section}

\begin{document}

\title{Commutators and images of noncommutative polynomials}

\author{Matej Bre\v sar}
\address{ Faculty of Mathematics and Physics,  University of Ljubljana,  and Faculty of Natural Sciences and Mathematics, University
of Maribor, Slovenia}
\email{matej.bresar@fmf.uni-lj.si}


\thanks{The author was supported by ARRS Grant P1-0288.}
\begin{abstract}
Let $A$ be an algebra and let $f$ be a nonconstant noncommutative polynomial. In the first part of the paper,
we consider the relationship between $[A,A]$, the linear span of commutators in $A$, and span$\,f(A)$,  the linear span of the image of $f$ in $A$.
In particular, we show that  $[A,A]=A$ implies  span\,$f(A)=A$. In the second part, we  establish some Waring type results for images of polynomials.
For example, we show that if $C$ is a 
commutative unital algebra over a field $F$ of characteristic $0$,
 $A$ is the matrix algebra $M_n(C)$, and the polynomial $f$ 
is neither an identity nor a central polynomial of $M_n(F)$, then  every commutator in $A$
can be written 
as a difference of two elements, each of which is a sum of $7788$ elements from $f(A)$ (if $C=F$ is an algebraically closed field, then
$4$ elements suffice).  Similar results 
are obtained for some other algebras, in particular for  the algebra $B(H)$ of all bounded linear operators on a Hilbert space $H$.
\end{abstract}

\maketitle

\section{Introduction}

Let $F$ be a field and let $F\langle \mathcal X\rangle$ be the free algebra generated by the set $\mathcal X = \{X_1,X_2,\dots\}$, i.e., the algebra of  noncommutative polynomials in the variables $X_i$. For any $F$-algebra $A$ and
$f=f(X_1,\dots,X_m)\in F\langle \mathcal X\rangle$, we set
$$f(A)=\{f(a_1,\dots,a_m)\,|\,a_1,\dots,a_m\in A\}$$
(we tacitly assume that $f$ has zero constant term if $A$ is not unital). We call $f(A)$ the {\em image of $f$}.
% Our main interest will be in span$\,f(A)$, the linear span of $f(A)$. In the special case where $f$ is the commutator $[X_1,X_2] =X_1X_2-X_2X_1$, we write, as usual, $[A,A]$ for span$\,f(A)$.

In Section \ref{s3},   we first consider the question of  when  span$\,f(A)$, the linear span of $f(A)$, equals $A$. The special case where $f$ is the commutator $[X_1,X_2]=X_1X_2-X_2X_1$ has been studied by numerous authors in different contexts (see, e.g., \cite{M, W}). We prove that the general case reduces to this special case (Theorem \ref{t}). 
This gives a substantial generalization of some results from \cite{BK2} and \cite{R} which establish the same conclusion for certain $C^*$-algebras.
We  also obtain  fairly conclusive results
 concerning the more general question of when span$\,f(A)$
contains $[A,A]$ (Theorems \ref{T2} and \ref{t3}).

Section \ref{s4} gives a  more detailed analysis of representing the elements of $A$ through the elements of $f(A)$. More specifically, we show that, under suitable assumptions, every element (or at least every commutator) in $A$ is a sum/difference or a linear combination of  a fixed number of elements from $f(A)$.
% Our results were partially motivated by
 One of our motivations  for such results was the theory of 
expressing the elements from $B(H)$, the algebra of all bounded linear operators on a  Hilbert space $H$, 
%as sums or linear combinations of 
by some special-type elements.
For example, if $H$ is infinite-dimensional, 
then every element of $B(H)$ is a sum of two commutators \cite{Hal}, a sum of five idempotents \cite{PT}, a linear combination of ten projections \cite{Mat}, a sum of five square-zero elements \cite{PT}, etc. (for a nice, although not updated, survey, see \cite{Wu}). These are just sample results, most of them  have  both predecessors (e.g., involving larger  numbers of these special-type elements) and successors (e.g., involving more general operator algebras). The approach we take is to consider elements from the image of a polynomial as the special-type ones. We remark that the idea for such an approach appears also in \cite{R}, but the results therein are of a somewhat different nature.

Our first main result of Section \ref{s4} (Theorem \ref{tt2}) states that if $F$ is a field satisfying small restrictions, 
the polynomial $f$ is   neither an identity nor a central polynomial of the algebra $M_n(F)$,
	and  $B$ is a unital $F$-algebra such that every element of the algebra $A=M_n(B)$ is a sum of $k$ commutators and a central element, then 
	  every commutator in $A$ can be written as a difference of two elements, each of which is a sum of $1936k^2 + 22k$ elements from $f(A)$. 
%We are primarily interested in concrete situations to which this theorem applies.
This applies to the case where
$A=M_n(C)$ with $C$ a commutative algebra (Corollary \ref{ts22}),   the case where  $A={\rm End}_F(V)$ (Corollary \ref{cend}), and also the case where $A=B(H)$ (Corollary \ref{cH}). In 
 the fundamental case where $A=M_n(F)$ with $F$ algebraically closed and of characteristic $0$,  the needed number of summands is substantially smaller 
 (Corollaries \ref{ts2} and \ref{cs2}). This is obtained as a consequence of our second main result  (Theorem \ref{secondmain}) which states that, under appropriate assumptions,
every square-zero matrix is a difference of two matrices from $f(A)$.

To the best of our knowledge, results of the type just described are new even for multilinear polynomials. This is interesting 
 in light of the L'vov-Kaplansky conjecture which states that if
$F$ is an infinite field  and 
 $f$ is a multilinear polynomial that is neither an identity  nor a central polynomial of $A=M_n(F)$, then $f(A)$ is equal to either $[A,A]$ or $A$. Interest in this old problem has been revived
 rather recently, starting with the work \cite{KBMR}
by Kanel-Belov, Malev, and Rowen  in which they confirmed the conjecture for the case where $n=2$ and $F$ is quadratically closed. Since then, quite a few papers studying this  and related problems appeared; for details, 
we refer %the reader
 to the survey paper \cite{Retc}. In spite of many efforts, however,  the problem seems to be far from being resolved.
Some of our results from Section \ref{s4} can  be viewed as
  rough approximate versions of the  L'vov-Kaplansky conjecture, or at least as an attempt to approach this problem from a different perspective. 
	
Last but not least, we point out that  the central theorems of Section \ref{s4}  and their corollaries may be viewed	as Waring type results, analogous to those on images of words that have been studied thoroughly  (and successfully)  by group theorists. 
We will  discuss this at the 
end of the paper.

The main ingredients in the proofs of Section \ref{s4}  are the PI theory,  the Lie theory of associative rings, and the theory of expressing commutators through square-zero elements. It is not 
surprising that these theories can be useful in the study of commutators and noncommutative polynomials. 
 The idea to combine them in proofs, however, seems to be new.

\section{On the relation between $[A,A]$ and {\rm span}\,$f(A)$}\label{s3}
%The containment of commutators in the linear spans of images of polynomials}\label{s3}

This section is divided into three subsections. The first one contains an important lemma and  some other preliminaries.

\subsection{Basic lemma}\label{sub21}
If $U$ and $V$ are subspaces of  an algebra $A$, we write  $[U,V]$ for the linear span of all commutators
$[u,v]$ with $u\in U$ and $v\in V$. 
Recall that a linear subspace $L$ of $A$ is a called a {\em Lie ideal} of $A$ if  
$[L,A]\subseteq L$.
 In \cite{BK1}, it was noticed that, as long as the field $F$ is infinite,  span$\,f(A)$ is  a Lie ideal of $A$ for every
polynomial  $f\in F\langle \mathcal X\rangle$ (and hence, in the case where $A=M_n(F)$, it is equal to either  $\{0\}$,  the set $F1$ of scalar matrices,  the set
$[A,A]$ of traceless matrices, or  $A$).
 This was then used in problems originally arising from Connes' embedding conjecture and some other topics of  functional analytic flavor. Our goals in the present paper are different, but this observation from \cite{BK1}
will be our starting point here as well. More precisely, we will use it in the proof of our first lemma.  Before stating it,  we 
record a few more definitions and  preliminary observations.

Let $f=f(X_1,\dots,X_m)\in F\langle \mathcal X\rangle$ be any polynomial. Recall that $f$ is a {\em (polynomial) identity} of the $F$-algebra $A$ 
if $f(A) = \{0\}$.  We say $A$ is a {\em PI-algebra} if a nonzero polynomial is an identity of $A$.
If $f(A)$ is contained in the center of $A$ but $f$ is not an identity of $A$, then we say that $f$ is a {\em central polynomial} of $A$ (one often assumes that central polynomials must have zero constant term, but we will avoid this additional requirement in order to make the exposition simpler).

 For the purposes of this paper, we define 
%For any  polynomial  $f=f(X_1,\dots,X_n)\in F\langle \mathcal X\rangle$, we define the polynomial
$\widehat{f}  \in F\langle \mathcal X\rangle$ by
$$\widehat{f}(X_1,\dots,X_{2m})= \big[f(X_1,\dots,X_m),f(X_{m+1},\dots,X_{2m})\big].$$
Note that $\widehat{f}$ is an identity of $A$ if and only if $f(A)$ is a commutative set. 

Finally, we point out that, if $A$ is unital,  $f(A)$ is  invariant under conjugation, i.e.,
\begin{equation*} \label{conj}
af(A)a^{-1}= f(A)
\end{equation*}
for every invertible element $a\in A$. Indeed, this follows from
%\begin{equation} \label{conj2}
$$af(a_1,\dots,a_m)a^{-1} = f(aa_1a^{-1},\dots,aa_ma^{-1}),$$
%\end{equation}
where $a_1,\dots,a_m$ are arbitrary elements in $A$.

The  following lemma, which is an outcome of Herstein's Lie theory of associative rings, will be used in several proofs.

\begin{lemma}\label{lt3}
Let $A$ be an algebra over an infinite field $F$ and let 
 $f\in F\langle \mathcal X\rangle$.
\begin{enumerate}
\item[{\rm (a)}] If
   $[A,A]\not\subseteq {\rm span}\,f(A)$, then $\widehat{f}$ is an identity of a nonzero homomorphic image of $A$.
%the ideal generated by  $\widehat{f}(A)$ is proper.
\item[{\rm (b)}] If $A$ is simple and $\widehat{f}$ is an identity of $A$, then $f$ is either an identity or a central polynomial of $A$.
\end{enumerate}
\end{lemma}

\begin{proof}   As  mentioned at the beginning of the section,  $L={\rm span}\,f(A)$ is a Lie ideal of $A$. %(specifically, this is is \cite[Theorem 2.3]{BK1}).

(a) Let $I$ be the ideal of $A$ generated by $\widehat{f}(A)$ 
  (i.e., the ideal generated by $[L,L]$). A general result on Lie ideals, essentially due to Herstein \cite[pp. 4--5]{Her} and explicitly established in \cite[Proposition 2.2]{BKS}, states that $L$ contains $[A,I]$.
Since $[A,A]\not\subseteq L$ by assumption, $I$ is  a proper ideal of $A$. Observe that $\widehat{f}(A)\subseteq I$ implies that  $\widehat{f}$ is an identity of $A/I$.

(b) The condition that $\widehat{f}$ is an identity of $A$ means  that $L$ is a commutative Lie ideal of $A$. Our goal is to prove that $L$ is contained in the center $Z$ of $A$. This follows  from  the two results stated on page 9 of  \cite{Her} (or, more directly, from \cite[Theorem 4]{LM}), unless $F$ has characteristic $2$ and $A$ is $4$-dimensional over $Z$. This case is indeed exceptional (see \cite[p.\,6]{Her}), but not in our situation as we will show in the next paragraph.

Thus, assume that  $F$ has characteristic $2$ and that $A$ is of dimension  $4$ over $Z$. 
Let $K$ be the algebraic closure of  $Z$ and let
 $S = K\otimes_Z A$ be the scalar extension of $A$ to $K$. Note that
$S\cong M_2(K)$ by Wedderburn's Theorem,  and that
  $\widehat{f}$ is an identity of $S$ (see \cite[Theorem 6.29]{INCA}). Thus,
$V={\rm span}\,f(M_2(K))$ is a commutative Lie ideal of $M_2(K)$,
which, in view of the remark preceding the statement of the lemma, additionally satisfies
%for any invertible matrix $a$ and any matrices $a_i$ we see that $V$ additionally satisfies
$aVa^{-1}= V$ for every invertible  $a\in M_2(K)$. We claim that this implies that $V$ is contained in the center of $M_2(K)$. Suppose this is not true, so that $V$ contains a nonscalar matrix $v$. If $v$ is a diagonal matrix, then, by commuting it with the matrix unit $e_{12}$, we see that $V$ also contains nondiagonal matrices. Without
loss of generality, we may assume that $v=\left[ \begin{smallmatrix} \alpha & \beta \cr \gamma & \delta \cr \end{smallmatrix} \right]$ with $\beta\ne 0$. Commuting $v$ with $e_{11}$, we see that 
$w=\left[ \begin{smallmatrix} 0 & \beta \cr \gamma & 0 \cr \end{smallmatrix} \right]\in V$, and commuting $w$ with $e_{21}$ we see that $V$ contains  the identity matrix. Further,
$$\left[ \begin{matrix} 1 & 1 \cr 0 & 1 \cr \end{matrix} \right]w\left[ \begin{matrix} 1 & 1 \cr 0 & 1 \cr \end{matrix} \right]^{-1} =\left[ \begin{matrix} \gamma & \beta+\gamma \cr \gamma & \gamma \cr \end{matrix} \right]\in V,$$ from which one easily infers that
 $V$ contains the matrix unit $e_{12}$. Hence, 
$$\left[ \begin{matrix} 0 & 1 \cr 1 & 0 \cr \end{matrix} \right]e_{12}\left[ \begin{matrix} 0 & 1 \cr 1 & 0 \cr \end{matrix} \right]^{-1} = e_{21}\in V,$$ which 
 is a contradiction since $e_{12}$ and $e_{21}$ do not commute. Therefore, $V$ is contained in the center of $M_2(K)$. This readily implies that $f(A)\subseteq Z$.
\end{proof}

% The next example justifies, at least partially, the assumption that $A$ has no nonzero nil homomorphic images.

%If a polynomial $f$ can be written as a sum of  commutators of some polynomials, then its image $f(A)$ obviously lies in $[A,A]$. The following corollary of
%Theorem \ref{t3}  therefore holds. 

%\begin{corollary}\label{c3}
%Let $A$ be a unital  algebra over an infinite field $F$.  Then  the following two statements are equivalent
%for every polynomial $f\in F\langle \mathcal X\rangle$ which is a sum of commutators in $ F\langle \mathcal X\rangle$:
%\begin{enumerate}
%\item[{\rm (i)}] $f$ is neither an identity nor a central polynomial on any nonzero homomorphic image of $A$.
%\item[{\rm (ii)}] $[A,A]={\rm span}\,f(A)$ and the ideal generated by $f(A)$ equals $A$.
%\end{enumerate}
%\end{corollary}

\subsection{When is ${\rm span}\,f(A)$ equal to $A$?}
Our first theorem %is  an application of Lemma \ref{lt3}. It
shows that 
the 
 condition  that ${\rm span}\,f(A)=A$ holds for any  nonconstant polynomial  is equivalent to the condition that this holds for the particular polynomial $[X_1,X_2]$.

\begin{theorem}\label{t}
Let $A$ be an algebra over an infinite field $F$. If  $[A,A]=A$, then ${\rm span}\,f(A)=A$ for every nonconstant polynomial $f\in F\langle \mathcal X\rangle$.
\end{theorem}

\begin{proof} 
In \cite{KB}, Kanel-Belov proved that no nonzero   PI-algebra $R$ coincides with $[R,R]$. Therefore, $[A,A]=A$ implies that  no nonzero homomorphic image of $A$ is a PI-algebra.
The desired conclusion thus follows from
Lemma \ref{lt3}\,(a).
\end{proof}

Theorem \ref{t} covers the main results of \cite{BK2} as well as, along with a theorem on commutators from \cite{Pop}, \cite[Corollary 3.10]{R}
(this result,   however, gives an additional information on the number of summands). 

\subsection{When does ${\rm span}\,f(A)$ contain $[A,A]$?}
We start this subsection with a variation  of Theorem \ref{t}.

\begin{theorem}\label{T2}
Let $A$ be a unital algebra over a field $F$ with {\rm char}$(F)=0$. If  $1\in [A,A]$,
 then $[A,A]\subseteq {\rm span}\,f(A)$ for every nonconstant polynomial $f\in F\langle \mathcal X\rangle$.
\end{theorem}

\begin{proof}
Suppose the theorem does not hold. Then,
 by
 Lemma \ref{lt3}\,(a), there exists a proper ideal $I$ of $A$ such that $A/I$ is a PI-algebra.
Choose a maximal ideal $M$ of $A$  containing $I$.
 Then $R=A/M$ is a simple PI-algebra.
Hence, $R$ is finite-dimensional over its center $Z$ \cite[Lemma 7.53]{INCA}.
Obviously, $1\in [A,A]$ implies $1\in [R,R]$. Let $K$ be the algebraic closure of $Z$ and let 
 $S = K\otimes_Z R$ be the scalar extension of $R$ to $K$. Note that
$1\in [S ,S]$. However, $S$ is isomorphic to $M_n(K)$ for some $n\ge 1$,  so we may consider  the trace of elements in $S$:
 the trace of $1$ is $n\ne 0$ (since char$(F)=0$), while the trace of any element in $ [S ,S]$ is $0$. This contradiction shows that $I=A$.
\end{proof}

\begin{example}
 Any Weyl algebra over a field of characteristic $0$ obviously satisfies 
 the conditions of  Theorem \ref{T2}  (as a matter of fact,
it also  satisfies the conditions of Theorem \ref{t}---see, e.g., \cite[Proposition 7]{M}). Note also that if  an algebra $A$  satisfies the conditions of 
 Theorem \ref{T2},
then so does $A\otimes B$ for any unital algebra $B$. 
\end{example}

The next example  shows that the assumption that char$(F)=0$  is necessary in Theorem \ref{T2}.

\begin{example}
Suppose  char$(F)$ is a prime number $p$. Let $A=M_p(F)$. Then $1$ has trace $0$ and is hence a commutator \cite{AM}. However, $[A,A]\not\subseteq {\rm span}\,f(A)$ if $f$ is either an identity or a  central polynomial of $A$.
\end{example}

 %In the first theorem derived from Lemma \ref{lt3} we provide conditions under which
We conclude this section with two results that give a more complete picture of when 
 ${\rm span}\,f(A)$ contains $[A,A]$. First we record a corollary to Lemma \ref{lt3}.

\begin{corollary}\label{lt3c}
Let $A$ be a simple algebra over an infinite field $F$. If  
 $f\in F\langle \mathcal X\rangle$  is 
	neither an identity nor a central polynomial of $A$, then
   $[A,A]\subseteq {\rm span}\,f(A)$.
\end{corollary}

Our last theorem in this section is a sharpening of this corollary.
Recall that the commutator ideal $C_A$ of $A$ is the 
ideal  generated by all commutators in $A$.

\begin{theorem}\label{t3}
Let $A$ be a unital algebra over an infinite field $F$, and let $f\in F\langle \mathcal X\rangle$. 
 The following two statements are equivalent:
\begin{enumerate}
\item[{\rm (i)}] $f$ is neither an identity nor a central polynomial of any nonzero homomorphic image of $A$.
\item[{\rm (ii)}] $[A,A]\subseteq {\rm span}\,f(A)$ and $A$ is equal to its commutator ideal $C_A$.
\end{enumerate}
\end{theorem}

\begin{proof}
(i)$\implies$(ii). 
 As $A/C_A$ 
 is a commutative algebra, every polynomial is either an identity or a central polynomial of $A/C_A$.
Thus,
it is enough to show that  $[A,A]\subseteq {\rm span}\,f(A)$.
Assume that this is not true. Then, by
Lemma \ref{lt3}\,(a), there is a proper ideal $I$ of $A$ such that $\widehat{f}$ is an identity of  $A/I$. Take a maximal ideal $M$ containing $I$. 
Then $\widehat{f}$  is an identity of the simple algebra $A/M$, and so  Lemma \ref{lt3}\,(b)  tells us that 
$f$ is either an identity or a central polynomial of $A/M$, which contradicts (i).

(ii)$\implies$(i). Suppose 
there exists a proper ideal $I$ of $A$ such that $f$ is an identity or a central polynomial of $A/I$.
This means that $[f(A),A]\subseteq I$, %  $$[f(a_1,\dots,a_m),a]\in I$$ for all $a_i,a\in A,$
 which along with 
$[A,A]\subseteq {\rm span}\,f(A)$  gives  $[[A,A],A]\subseteq I$. 
Pick a maximal ideal $M$ containing $I$. Then $R=A/M$ is a simple algebra satisfying
$[[R,R],R]=\{0\}$. That is, every commutator in $R$ lies in the center $Z$ of $R$;
in particular, $[x,y]x=[x,yx]\in Z$, and so $[x,y]^2 =[[x,y]x,y]=0$ for all $x,y\in R$. As the center of a simple  algebra cannot contain nonzero nilpotent elements, it follows that $R$ is commutative.
 That is, $[A,A]\subseteq M$, contradicting   $C_A=A$.
\end{proof}

We make two final comments. First, the necessity of the requirement in (ii) that $A=C_A$ is  evident from the case where $A$ is commutative. Second,
the statement (i) can be equivalently formulated as that the ideal generated by $[f(A),A]$ is equal to $A$; on the other hand, the statement (ii) concerns only the linear span of $f(A)$.

\section{Waring type results for images of polynomials}\label{s4}

This section consists of five subsections. %The main results are contained in the third one. 
The basic purpose of the first one  is to provide some tools needed in the proofs of the main results. However, as these tools are interesting in their own right,
 we will present them in a more general form than would be necessary for applications we have in mind.

\subsection{On locally linearly dependent polynomials}\label{sub31} 
Let $A$ be an algebra over a field $F$. We say that polynomials $f_1,\dots,f_s\in F\langle X_1,\dots,X_m\rangle$  are {\em $A$-locally linearly dependent} if for any 
$\overline{a}=(a_1,\dots,a_m)\in A^m$, the elements $f_1(\overline{a}),\dots,f_s(\overline{a})$ are linearly dependent over $F$ \cite{BK3}.
Otherwise, we say that they are $A$-locally linearly independent.

The usual linear dependence obviously implies the $A$-local linear dependence. The converse is not true. 
For example,  
 a single polynomial $f$ is $A$-locally linearly dependent if and only if $f$ is an identity of $A$. Similarly,
if the center of $A$ consists of scalar multiples of unity, then 
 $1,f$ are
 $A$-locally linearly dependent if and only if $f$ is either an identity or a central polynomial of $A$.
%if  a nonconstant polynomial $f$ is  central for $A$, then 
%the polynomials $1$ and $f$ are linearly independent, but $A$-locally linearly dependent.
 As another example, 
the linearly independent polynomials $1,X,\dots,X^n$
are $M_n(F)$-locally linearly dependent by
 the Cayley-Hamilton Theorem.  We remark, however, that the relation of the  
$M_n(F)$-local linear dependence can be viewed as  a  functional identity on $M_n(F)$,  which, in general, is not    a consequence of
 the Cayley-Hamilton Theorem; see \cite{BPS}.

Let $c_s$ denote the $s$th Capelli polynomial, i.e.,
$$c_s(X_1,\dots,X_s,Y_1,\dots,Y_{s-1})
= \sum_{\sigma\in S_s} {\rm sgn}(\sigma)X_{\sigma(1)}Y_1 X_{\sigma(2)}Y_2\dots Y_{s-1}X_{\sigma(s)}.
$$
We record two  remarks related to $c_s$.

\begin{remark}\label{re1}It is a standard fact that matrices $b_1,\dots,b_s\in A=M_n(F)$ are linearly dependent if and only if 
$$c_s(b_1,\dots,b_s,y_1,\dots,y_{s-1})=0$$
for all $y_1,\dots,y_{s-1}\in A$ \cite[Theorem 7.45]{INCA}.  Observe that this implies that
polynomials 
 $f_1,\dots,f_s$ are $A$-locally linearly dependent if and only if 
$$c_s(f_1,\dots,f_s,Y_1,\dots,Y_{s-1})$$ is an identity of $A$. 
\end{remark}

%We need one more remark. 

\begin{remark}\label{re2}
Let $F$ be an infinite field and let $A=M_n(F)$. 
Suppose $f_1,\dots,f_s,g_1,\dots,g_t\in F\langle X_1,\dots,X_m\rangle$ are such that
for every
$\overline{a}\in A^m$,
either $f_1(\overline{a}),\dots,f_s(\overline{a})$ or 
 $g_1(\overline{a}),\dots,g_t(\overline{a})$ are linearly dependent. We claim that then either 
$f_1,\dots,f_s$  or  
$g_1,\dots,g_t$ are $A$-locally linearly dependent. Indeed, %as mentioned in Remark \ref{re1}, 
our assumption can be stated as that either
$$c_{s}\big(f_1(\overline{a}),\dots,f_{s}(\overline{a}), y_1,\dots,y_{s-1}\big)=0$$
		for all $y_1,\dots,y_{s-1}\in A$, or
		$$c_{t}\big(g_1(\overline{a}),\dots,g_{t}(\overline{a}), y_1,\dots,y_{t-1}\big)=0$$
	for all $y_1,\dots,y_{t-1}\in A$. This implies that the product of the polynomials
	$$c_{s}\big(f_1,\dots,f_{s}, Y_1,\dots,Y_{s-1}\big)$$
	and 
	$$c_{t}\big(g_1,\dots,g_{t}, Y_1,\dots,Y_{t-1}\big)$$
	is an identity of $A$. 
We are thus in a position to apply the classical theorem by Amitsur \cite{A2} which
states that  the product of two polynomials is an identity of $M_n(F)$ (with $F$ infinite) only when  one of them is an identity. Our claim therefore follows from Remark \ref{re1}.
\end{remark}

Since the free algebra $ F\langle \mathcal X\rangle$ is a domain, the linear independence of  $f_1,\dots,f_s\in   F\langle \mathcal X\rangle$ implies the linear
independence of $hf_1,\dots,hf_s$ for every nonzero $h\in   F\langle \mathcal X\rangle$. We will now prove  a similar statement 
for the $M_n(F)$-local linear independence.

\begin{theorem}\label{int}
Let $F$ be an infinite field, let $A=M_n(F)$ with   $n\ge 2$, and let $h, f_1,\dots,f_s\in  F\langle X_1,\dots,X_m\rangle$.
  Assume that
$h$ is not an identity of $A$. 
If 
$f_1,\dots,f_s$ are $A$-locally linearly independent, then so are $hf_1,\dots,hf_s$.
\end{theorem}

\begin{proof}
The  $s=1$
case follows from  the aforementioned Amitsur's theorem.
%Amitsur's theorem which  states that  the product of two polynomials is an identity of $M_n(F)$ (with $F$ infinite) only when  one of them is an identity \cite{A2}.
We may therefore assume that $s>1$ and  $hf_1,\dots,hf_{s-1}$ are $A$-locally linearly independent. Assume  further that $hf_1,\dots,hf_{s}$ are $A$-locally linearly dependent, and let us show that this leads to a contradiction.
%Thus, there exists a $\overline{u} = (u_1,\dots,u_m)\in A^m$ such that $h(\overline{u})f_1(\overline{u}),\dots,h(\overline{u})f_{s-1}(\overline{u})$ are linearly independent over $F$.

%Let $c_s$ denote the $s$th Capelli polynomial. Since $$h(\overline{a})f_1(\overline{a}),\dots,h(\overline{a})f_s(\overline{a})$$ are linearly dependent
%for every $\overline{a} = (a_1,\dots,a_m)\in A^m$, we have 
%$$c_s\big(h(\overline{a})f_1(\overline{a}),\dots,h(\overline{a})f_s(\overline{a}),y_1,\dots,y_{s-1}\big)=0$$
%for all $y_1,\dots,y_{s-1}\in A$ \cite[Theorem 7.45]{INCA}. 
According to Remark \ref{re1},  the polynomial
$$c_s(hf_1,\dots,hf_s,Y_1,\dots,Y_{s-1})$$
	is an identity of $A$. That is,
	\begin{align*}&\sum_{i=1}^s (-1)^{i-1} hf_iY_1c_{s-1}(hf_1,\dots,hf_{i-1},hf_{i+1},\dots,hf_s,Y_2,\dots,Y_{s-1})\\
	=&h\cdot \sum_{i=1}^s (-1)^{i-1} f_iY_1c_{s-1}(hf_1,\dots,hf_{i-1},hf_{i+1},\dots,hf_s,Y_2,\dots,Y_{s-1})
	\end{align*}
	is an identity (if $s=2$, it should be understood that $c_1(X)=X$). Since 
	$h$ is not an identity, Amitsur's theorem implies that the latter factor,
\begin{equation*}g=\sum_{i=1}^s  f_iY_1 \Big((-1)^{i-1}c_{s-1}(hf_1,\dots,hf_{i-1},hf_{i+1},\dots,hf_s,Y_2,\dots,Y_{s-1})\Big),
\end{equation*}
		is an identity.

		Take  $\overline{a}\in A^m$. Suppose that $$h(\overline{a})f_1(\overline{a}),\dots,h(\overline{a})f_{s-1}(\overline{a})$$ are linearly independent.
By Remark \ref{re1},   there are $y_2,\dots,y_{s-1}\in A$ such that
		$$b_{s}= (-1)^{s-1}c_{s-1}\big(h(\overline{a})f_{1}\overline{a}),\dots,h(\overline{a})f_{s-1}(\overline{a}),y_2,\dots,y_{s-1}\big)\ne 0.$$
	Since the polynomial $g$ is an identity,	 we have
		$$ f_{1}(\overline{a})y_1b_1 + f_{2}(\overline{a})y_1b_2  +\dots + f_{s-1}(\overline{a})y_1b_{s-1} +   f_{s}(\overline{a})y_1b_{s} =0$$
		for all $y_1\in A$ and some $b_1,\dots,b_{s-1}\in A$. As $b_s\ne 0$, this implies that 
		$ f_{1}(\overline{a}),\dots,  f_{s}(\overline{a})$ are linearly dependent
		 \cite[Lemma 7.42]{INCA}.
		
		We have thus shown that for each $\overline{a} \in A^m$, either 
		$$h(\overline{a})f_1(\overline{a}),\dots,h(\overline{a})f_{s-1}(\overline{a})\,\,\,\,\mbox{or}\,\,\,\,
f_1(\overline{a}),\dots,f_{s}(\overline{a})$$
are linearly dependent.  Hence, by Remark \ref{re2}, either 
		$$ hf_1,\dots,hf_{s-1}\,\,\,\,\mbox{or}\,\,\,\, f_1,\dots,f_{s}$$
		are $A$-locally
		linearly dependent. However, this 
  contradicts our initial assumptions.
\end{proof} 
%Hence, by Remark \ref{re2}, either $$hf_1,\dots,hf_{s-1}$ 
%are
% $A$-locally linearly independent or
 %$f_1,\dots,f_s$  are $A$-locally linearly independent,

% Note that 
%		 In light of  \cite[Theorem 7.45]{INCA},
 %this can be equivalently stated as that either
	%	$$c_{s-1}\big(h(\overline{a})f_1(\overline{a}),\dots,h(\overline{a})f_{s-1}(\overline{a}), y_2,\dots,y_{s-1}\big)=0$$
%		for all $y_2,\dots,y_{s-1}\in A$ or
	%	$$c_{s}\big(f_1(\overline{a}),\dots,f_{s}(\overline{a}), y_1,\dots,y_{s-1}\big)=0$$
	%for all $y_1,\dots,y_{s-1}\in A$. This implies that the product of the polynomials
	%$$c_{s-1}\big(hf_1,\dots,hf_{s-1}, Y_2,\dots,Y_{s-1}\big)$$
	%and 
	%$$c_{s}\big(f_1,\dots,f_{s}, Y_1,\dots,Y_{s-1}\big)$$
	%is an identity of $A$. But then, by Amitsur's theorem, one of them is an identity. 
%Using  \cite[Theorem 7.45]{INCA} again, we see that
%However, this contradicts our assumption that
	%$hf_1,\dots,hf_{s-1}$ as well as $f_1,\dots,f_s$  are $A$-locally linearly independent.

%The following lemma follows from \cite[Theorem 2.4]{LZ}, but we will give a different proof.

\begin{corollary}\label{pl}
Let $F$ be an infinite field and let $A=M_n(F)$ with   $n\ge 2$.
 If $f\in  F\langle \mathcal X\rangle$  is not an identity of $A$, then 
there exists a positive integer  $k\le n$ such that $1,f,\dots,f^k$ are $A$-locally linearly dependent, but $f,\dots,f^k$ are $A$-locally linearly independent. In particular,
$f(A)$ contains an invertible matrix.
\end{corollary}

\begin{proof} The Cayley-Hamilton Theorem
tells us that  $1,f,\dots,f^n$ are $A$-locally linearly dependent. Let $k\le n$ be the smallest positive integer such that
 $1,f,\dots,f^k$ are $A$-locally linearly dependent. Then  $1,f,\dots,f^{k-1}$ are $A$-locally linearly independent, and hence, by Theorem \ref{int}, so are
$f,\dots,f^k$. Therefore, there exists an $a\in f(A)$ such that 
$\lambda_01 + \lambda_1 a + \dots + \lambda_k a^k =0$ for some $\lambda_i\in F$ with $\lambda_0\ne 0$. Thus, $a$ is invertible.
\end{proof}

\begin{remark}The last statement of this corollary was proved earlier, by a different method, by Lee and Zhou \cite[Theorem 2.4]{LZ}. We will use it in
the proofs of Theorems \ref{tder} and \ref{tt2}.
Incidentally, the assumption that $F$ is infinite, which will be present in our main results, is indispensable in this corollary---see \cite[Example on p.\,294]{C}.
 \end{remark}

If $f$ is a central polynomial for $A=M_n(F)$, then $f(A)$ consists only of invertible matrices and (possibly) $0$. The following example, borrowed from \cite{Sp},
shows that this may also hold  for  noncentral polynomials (for more profound
examples, see \cite[Corollary 3.3 and Proposition 4.1]{Sp}).

\begin{example}\label{ekk}
Let $f=[X_1,X_2]^3$ and $A=M_2(F)$. As is well-known, $[X_1,X_2]^2$ is a central polynomial for $M_2(F)$, in fact, $[a,b]^2 = -\det([a,b])1$ for all $a,b\in M_2(F)$.  
Therefore, $f(a,b)= -\det([a,b])[a,b]$, and so, besides $0$, $f(A)$ contains only invertible elements (none of which is a  scalar multiple of the identity if char$(F)\ne 2$).
\end{example}

This example also shows that the multiplicative semigroup generated by the image of a polynomial  (which is neither an identity nor central) in $M_n(F)$ may be a relatively small subset of  $M_n(F)$. Therefore, when seeking for analogues of Waring's problem for groups  in the context of associative algebras (see comments at the end of the paper), it is more appropriate to consider the linear span  of the image of a noncommutative polynomial, rather than the 
 multiplicative semigroup it generates.

Our next goal is to prove  a theorem concerning the algebraic multiplicity of eigenvalues of matrices in $f(A)$. 
But first, some preliminaries.  
%To this end, we need an elementary linear algebraic lemma. But first, we make some comments. 
By ad$_a$ we denote the linear map defined by ad$_a(x)=[a,x]$. Observe that 
$${\rm ad}_a^k(x) = \sum_{i=0}^{k} (-1)^i {k \choose i}a^{k-i} x a^{i}.$$ 
This shows that $u^s =0$ implies ${\rm ad}_u^{2s-1}=0$.

\begin{lemma}\label{lad}
Let $F$ be any field, let $A=M_n(F)$ with   $n\ge 2$, let $\frac{n}{2} < s \le n$, let $u\in M_s(F)$ be a nilpotent  matrix, and let $a\in A$ be a matrix of the form $ \left[ \begin{smallmatrix} u & 0 \cr 0 & * \cr \end{smallmatrix} \right]$. Then, for every $k\ge 2s-1$, ${\rm ad}_a^{k}(A)$ contains no invertible matrices.
\end{lemma}

\begin{proof}
Take any $x=  \left[ \begin{smallmatrix} x' & * \cr * & * \cr \end{smallmatrix} \right]\in A$ where $x'\in M_s(F)$.  An easy computation shows that, for any $r\ge 1$,
 $${\rm ad}_a^r(x) = \left[ \begin{matrix} {\rm ad}_u^r(x') & * \cr * & * \cr \end{matrix} \right].$$ As pointed out  above, $u^s=0$ implies  ${\rm ad}_u^{2s-1} =0$.
Therefore,  $${\rm ad}_a^{2s-1}(x) =  \left[ \begin{matrix} 0 & * \cr * & * \cr \end{matrix} \right].$$
This readily implies that 
$${\rm ad}_a^{k}(x) =  \left[ \begin{matrix} 0 & * \cr * & * \cr \end{matrix} \right]$$
for every $k\ge 2s-1$.
The upper-right corner matrix has size $s\times (n-s)$.
Since $s> n-s$,  its rows are linearly dependent. Therefore, ${\rm ad}_a^{k}(x)$ is not invertible in $A$.
\end{proof}

If  char$(F)=0$, then the observation  that $u^s=0$ implies  ${\rm ad}_u^{2s-1} =0$ has a converse. That is, if  $b \in A=M_n(F)$ and  $d\ge 1$ are  such that ${\rm ad}_b^{d} =0$ on $A$, then 
there exists a scalar matrix $\alpha$ such that $u=b-  \alpha$ satisfies $u^{[(d+1)/2]} =0$ \cite{Her2} (see also \cite{MM}). Note that this in particular shows that 
  ${\rm ad}_b^{2s} =0$ implies 
  ${\rm ad}_b^{2s-1} =0$. 
	
\begin{theorem}\label{tder}
Let $F$ be a  field  with {\rm char}$(F)=0$ and let $A=M_n(F)$ with   $n\ge 2$.
 If $f=f(X_1,\dots,X_m)\in  F\langle \mathcal X\rangle$  is neither an identity nor a central polynomial of $A$, then $f(A)$ contains a matrix such that the algebraic multiplicity of any of its eigenvalues does not exceed $\frac{n}{2}$.
\end{theorem}

\begin{proof}
Suppose the theorem is false. Then, for any
$b\in f(A)$ there exists an eigenvalue  $\lambda\in \overline{F}$, the algebraic closure of $F$, whose algebraic multiplicity is $s >\frac{n}{2}$. Let $p\in M_n(\overline{F})$
be an invertible matrix such that $pbp^{-1}$ is of the form $ \left[ \begin{smallmatrix} \lambda + u & 0 \cr 0 & * \cr \end{smallmatrix} \right]$ where $u\in M_s(F)$ is a nilpotent matrix. By Lemma \ref{lad}, 
$${\rm ad}_{pbp^{-1}}^{2n-1}(M_n(\overline{F})) = {\rm ad}_{pbp^{-1}-\lambda}^{2n-1}(M_n(\overline{F}))$$ 
contains no invertible matrices. Therefore, ${\rm ad}_{b}^{2n-1}(A)$ also does not contain invertible matrices. That is to say, none of the matrices in the image of the polynomial
$${\rm ad}_{f}^{2n-1}(X_{m+1})= [f,[\dots[f,[f,X_{m+1}]]\dots]]$$
is invertible in $A$. Hence, ${\rm ad}_{f}^{2n-1}(X_{m+1})$ is an identity of $A$ by Corollary \ref{pl}. 

Let $d$ be the smallest positive integer such that
 ${\rm ad}_{f}^{d}(X_{m+1})$ is an identity of $A$. That is,  ${\rm ad}_{b}^{d}=0$ on $A$ for every $b\in f(A)$. 
We are now in a position to use the aforementioned Herstein's result from \cite{Her2}. Note, first of all, that the remark preceding the statement of the theorem tells us that $d$ is odd; moreover, 
$d\ne 1$ 
since $f$ is neither an identity nor a central polynomial of $A$. By \cite{Her2}, for any $b\in f(A)$ there exists a scalar matrix $\alpha$ such that
$v=b-\alpha$ satisfies $v^{(d+1)/2}=0$. This implies that
$${\rm ad}_b^{d-1}(x) = {\rm ad}_v^{d-1}(x) = (-1)^{(d-1)/2} {d-1 \choose (d-1)/2}v^{(d-1)/2} x v^{(d-1)/2}$$ 
for every $x\in A$. Therefore, ${\rm ad}_b^{d-1}(x)$ is not invertible. In other words, none of the matrices in the image of the polynomial
${\rm ad}_f^{d-1}(X_{m+1})$ is invertible. Hence, Corollary \ref{pl} tells us that ${\rm ad}_f^{d-1}(X_{m+1})$
 is an identity, which contradicts the choice of $d$.
\end{proof}

\begin{remark} \label{r2c}
A polynomial $f$ is said to be {\em $2$-central} for $M_n(F)$ if $f ^2$ is central, but $f$ is not. The simplest example is $[X_1,X_2]$ which is $2$-central for $M_2(F)$. It turns out that $2$-central polynomials for $M_n(F)$ exist for various even $n$ (see, e.g., \cite[Proposition 4.10]{Retc}). 
Obviously, any matrix in the image of  such a polynomial has at most two eigenvalues. On the other hand,
from 
Theorem \ref{tder} we infer  that  the image  always contains a matrix having two eigenvalues of algebraic multiplicity exactly $\frac{n}{2}$.
\end{remark}

\subsection{A lemma on sets invariant under conjugation}\label{sub32}
The next lemma reveals the main idea upon which this section is based. We will need it in the case where $T=f(A)$, %(which fulfills  the basic assumption, see \eqref{conj}), 
but maybe it has some independent interest.

\begin{lemma}\label{ls}Let $F$ be a field with {\rm char}$(F)\ne 2$, %let  $f\in F\langle \mathcal X\rangle$ be any polynomial, 
let $B$ be a unital $F$-algebra, and let $A=M_n(B)$ with $n\ge 2$. 
 % such that every element in $A=M_n(B)$ is a sum of $k$ commutators and a central element. 
 If a subset $T$ of $A$ is invariant under conjugation by invertible elements in $A$, then:
\begin{enumerate} 
\item[{\rm (a)}] Any element of the form
 $[t,u]$, where $t\in T$ and $u$ is a square-zero element in $A$, lies in $T-T=\{t-t'\,|\,t,t'\in T\}$.
% is a difference of two elements from $T$.
\item[{\rm (b)}] Any element of the form
 $[t,[x,y]]$, with $t\in T$ and $x,y\in A$,
 is  a sum of $22$ elements from $T-T$.
%difference of two elements, each of which is a sum of $22$ elements from $T$.
\item[{\rm (c)}] 
If
 $T$ contains elements 
$t_1,t_2$ such that  $[t_1,t_2]$ is invertible in $A$ and, for some $k\ge 1$,
every element in $A$ is a sum of $k$ commutators  and a central element, then every commutator in $A$ is  a %difference of two elements, each of which is a 
sum of $1936k^2 + 22k$ elements from $T-T$.

\item[{\rm (d)}]  If $B=F$ is an algebraically closed field and $T$ contains a matrix $t$ such that the
 algebraic multiplicity of any of its eigenvalues does not exceed $\frac{n}{2}$, then all square-zero elements in $A$ lie in $T-T$.
\end{enumerate}
\end{lemma}

\begin{proof}
%Let $a_i\in A$ be such that 
%$t=f(a_1,\dots,a_m)$.
%Assume first that $x^2=0$. %Then, with reference to \eqref{conj2}, we have
(a) Since $u^2 =0$, $(1-\frac{u}{2})^{-1}= 1+\frac{u}{2}$. Therefore,
\begin{align*}[t,u] =
 \Big(1-\frac{u}{2}\Big)t\Big(1-\frac{u}{2}\Big)^{-1} - \Big(1-\frac{u}{2}\Big)^{-1}t\Big(1-\frac{u}{2}\Big),
\end{align*}
which shows that $[t,u]\in T-T$.

(b) This follows from (a) and 
% The desired conclusion therefore   follows from 
 \cite[Theorem 4.4]{ABESV} which states that  every commutator in $A$ is a sum of $22$  square-zero elements.

(c)
The  proof  is based on the  identity
\begin{equation*}\label{e1}
[w,z] =  \big[[t_2,w[t_1,t_2]^{-1}],t_1z\big] - \big[[t_2,w[t_1,t_2]^{-1}t_1],z\big] + \big[t_1,z[t_2,w[t_1,t_2]^{-1}]\big],
\end{equation*}
which one can check by a direct calculation. %\footnote{Similarly we can express $[x[t_1,t_2]z,y]$ which should yield results on the number of summands in expressing commutators through $f(A)$,  similar to those in \cite{R}.}
This identity shows that the commutator  $[w,z]$  of any two elements $w$ and $z$ of $A$ can be written as %a sum 
%of commutators of the form 
$$[[t_2,x_1],x_2] + [[t_2,x_3],x_4] + [t_1,x_5]$$
for some $t_i\in T$ and $x_i\in A$. 
Using (b), along with our assumption that
every element in $A$ is a sum of $k$ commutators and a central element, we see that
every commutator $[t,x]$, with $t\in T$ and $x\in A$, is
%is a difference of two elements, each of which is 
a sum of $22k$ elements from $T-T$.
Consequently, every commutator $[w,z]$  is  a sum of $$2\cdot (22k)^2 + 2\cdot (22k)^2  +22k = 1936k^2 + 22k$$
elements from $T-T$.

(d)
 Let $\lambda_1,\dots,\lambda_r$, $r\ge 2$, be the distinct eigenvalues of $t$. 
Since $T$ is invariant under conjugation, we may assume that 
$$
t=\left[\begin{matrix} t_{\lambda_1} & 0 &\dots & 0 \cr
0 & t_{\lambda_2} &\dots & 0 \cr
\vdots & \vdots & \ddots & \vdots\cr
0 & 0 &\dots &t_{\lambda_r} \cr
\end{matrix} \right],
$$  
where
$t_{\lambda_i}$ is an upper triangular matrix having $\lambda_i$ on its diagonal. By assumption,
the size of any $ t_{\lambda_i}$ does not exceed $\frac{n}{2}$. Without loss of generality, we
 may  assume that $ t_{\lambda_1}$ and $ t_{\lambda_r}$
have larger or equal size than $ t_{\lambda_2},\dots, t_{\lambda_{r-1}}$.

% If $u\in A$ has square zero, then
%$[t,u]$ lies in $f(A)-f(A)$ by Lemma \ref{ls}\,(a). Now, choose 
Let
$$u = \left[ \begin{matrix} 0 & d \cr 0 & 0 \cr \end{matrix} \right]$$
where the  upper-left (resp. lower-right) corner matrix has size $[\frac{n}{2}]\times [\frac{n}{2}]$ (resp. $[\frac{n+1}{2}]\times [\frac{n+1}{2}]$)  and $d$ is the $[\frac{n}{2}]\times [\frac{n+1}{2}]$ matrix having arbitrary entries $d_i$ on the main diagonal and zeros elsewhere (the last column of $d$ is thus zero if $n$ is odd).  Write accordingly
$$t = \left[ \begin{matrix} t_1 & * \cr 0 & t_2 \cr \end{matrix} \right],$$
where $t_1$ is of size $[\frac{n}{2}]\times [\frac{n}{2}]$ and $t_2$ is of size 
$[\frac{n+1}{2}]\times [\frac{n+1}{2}]$. Specifically,
$$
t_1=\left[\begin{matrix} t_{\lambda_1} & \dots  & 0& 0&\cr
\vdots & \ddots & \vdots & \vdots\cr
0 & \dots & t_{\lambda_{j-1}} & 0 \cr
0 & \dots &0 &t_{\lambda_j}' \cr
\end{matrix} \right],\quad\quad
t_2=\left[\begin{matrix} t_{\lambda_j}'' & 0 &\dots & 0 \cr
0 & t_{\lambda_{j+1}} &\dots & 0 \cr
\vdots & \vdots & \ddots & \vdots\cr
0 & 0 &\dots &t_{\lambda_r} \cr
\end{matrix} \right],
$$
 where 
$$t_{\lambda_j} = \left[ \begin{matrix} t_{\lambda_j}' & * \cr 0 & t_{\lambda_j}'' \cr \end{matrix} \right]$$
($t_{\lambda_j}''$ may be zero; in fact, if $r=2$, then both $t_{\lambda_j}'$ and $t_{\lambda_j}''$ are zero). We have
$$[t,u]=  \left[ \begin{matrix} 0 & t_1d-dt_2 \cr 0 & 0 \cr \end{matrix} \right]$$
and $t_1d-dt_2$ is a matrix of size  $[\frac{n}{2}]\times [\frac{n+1}{2}]$ with zeros below the main diagonal and  main diagonal entries equal to the product of $d_i$ and
the difference of two distinct eigenvalues of $t$. Indeed, this is because $\lambda_1,\dots, \lambda_j,\dots,\lambda_r$ are distinct and  the size of
$t_{\lambda_j}$, and hence also the size of $t_{\lambda_j}'$ and $t_{\lambda_j}''$, does not exceed the size of  $ t_{\lambda_1}$ and $ t_{\lambda_r}$. 
Another property that we need is  that $d_i=0$, $i=k,\dots,[\frac{n}{2}]$, implies that the $i$th row of $t_1d-dt_2$, $i=k,\dots,[\frac{n}{2}]$, 
 is zero. 

All this shows that we may  choose
$d$ in such a way that $t_1d-dt_2$ has any rank between $0$ and $[\frac{n}{2}]$. Obviously, $[t,u]$ is a square-zero matrix
of the same rank as $t_1d-dt_2$. Note that $u^2=0$ and so $[t,u]$ lies in $T-T$ by \,(a).
 Since any square-zero matrix in $A$ has rank at most  $[\frac{n}{2}]$ and two square-zero matrices have the same rank if and only if they are similar,
the desired conclusion that $T-T$ contains all square-zero matrices follows from the fact that this set is invariant under conjugation by invertible matrices. 
\end{proof}

\begin{remark} \label{Paz}
The number $22$ in (b) may not be the smallest possible. In particular, if $B=F$, then it can be replaced
 by $4$ \cite[Theorem 1.1]{dP}. Accordingly, the number $1936k^2 + 22k$  in (c) can be, in this case, replaced by $2\cdot 4^2 + 2\cdot 4^2 + 4=68$,  provided that char$(F)$ is either $0$ or does not divide $n$. This is because every
matrix with zero trace is a commutator \cite{AM} and so 
$k=1$ under these assumptions.  
 \end{remark}

\subsection{Main theorems and their corollaries}
We will now consider the situation where $T=f(A)$. As above, we write
$$f(A)-f(A)=\{t-t'\,|\,t,t'\in f(A)\}.$$
By combining some results of Subsections \ref{sub21}, \ref{sub31}, and \ref{sub32},
we can now prove our first main result.

\begin{theorem}\label{tt2}Let $F$ be an infinite field with {\rm char}$(F)\ne 2$, let $n\ge 2$, 
  let $f=f(X_1,\dots,X_m)\in F\langle \mathcal X\rangle$ be a   polynomial  which is neither an identity nor a central polynomial of $M_n(F)$, let $k\ge 1$,
	and let $B$ be a unital $F$-algebra such that every element in $A=M_n(B)$ is a sum of $k$ commutators and a central element. 
	 Then every commutator in $A$ is %a difference of two elements, each of which is
	a sum of $1936k^2 + 22k$ elements from $f(A)-f(A)$. 
\end{theorem}

\begin{proof} By Lemma \ref{lt3}\,(b), the polynomial $\widehat{f}$ is not an identity of  $M_n(F)$.  
Corollary \ref{pl} therefore tells us that $\widehat{f}(M_n(F))$ contains an invertible matrix. That is to say, there exist $t_1,t_2\in f(M_n(F))$ such that $[t_1,t_2]$ is invertible in $M_n(F)$. Since $M_n(F)$ is a (unital) subalgebra of $A$, we may regard $t_1$ and $t_2$ as elements of $f(A)$ whose commutator $[t_1,t_2]$ is invertible in $A$. Finally, recall from the introductory comments in Section \ref{s3} that $f(A)$ is invariant under conjugation.
 The theorem therefore follows from Lemma \ref{ls}\,(c) applied to  $T=f(A)$.
\end{proof}

% Finally, use the fact that  every traceless matrix $t$ can be written as a commutator $[u,y]$ \cite{AM}.

\begin{corollary}\label{ts22}Let $F$ be field of characteristic $0$,  let $f\in F\langle \mathcal X\rangle$ be a   polynomial  which is neither an identity nor a central polynomial of $M_n(F)$, and
let $C$ be a commutative unital $F$-algebra.
 Then every commutator in $A=M_n(C)$ is
%a difference of two elements, each of which is
 a sum of $7788$ elements from $f(A)-f(A)$. 
\end{corollary}

\begin{proof} Since $F$ has characteristic $0$,  we can write every $x\in A$ as the sum of the traceless matrix $x - \frac{{\rm tr}(x)}{n}1$ and the matrix 
$\frac{{\rm tr}(x)}{n}1$ which lies in the center of $A$. By \cite[Theorem 15]{M}, every traceless matrix is a sum of two commutators.
Therefore, 
 the conditions of Theorem \ref{tt2} are met for $k=2$.
\end{proof}

\begin{remark}\label{r6}
For many commutative algebras $C$, including all principal integral domains \cite{S}, every traceless matrix in $M_n(C)$ is actually a commutator. Then the $k=1$ case 
of Theorem \ref{tt2} applies,  so the number $7788$ can be replaced by $1958$. If $C=F$, then we see from Remark \ref{Paz} that this number can be further lowered to 68. However, see also Corollary \ref{ts2} below.
\end{remark}

%\begin{proof} 
%The only difference with the proof of the preceding corollary is that now every traceless matrix is a commutator \cite{AM}, so the $k=1$ case of Theorem \ref{tt2} applies.
%\end{proof}

%This corollary obviously holds for the algebra of matrices  over any commutative unital algebra $C$ with the property that all traceless matrices over $C$ are commutators. For example, every principal integral domain $C$ has this property  \cite{S}.  

\begin{corollary} \label{cend} Let  $F$ be an infinite field with {\rm char}$(F)\ne 2$, let $V$ be an infinite-dimensional vector  space over $F$, and
let  $f\in F\langle\mathcal X\rangle$ be a  nonconstant polynomial. Then every element in
 $A={\rm End}_F(V)$ is 
%a difference of two elements, each of which is 
a sum of $1958$ elements from $f(A)-f(A)$. 
\end{corollary}

\begin{proof}
Pick an $n\ge 2$ such that $f$  is neither an identity nor a central polynomial of $M_n(F)$ (see \cite[Lemma 6.38]{INCA}). As
$V$ is isomorphic to $V^n$, the direct sum of $n$ copies of $V$,  End$_F(V)$ is isomorphic to $M_n({\rm End}_F(V))$. Since every element in End$(_F(V))$  is a commutator
 \cite[Proposition 12]{M}, we may apply Theorem \ref{tt2} for $k=1$.
\end{proof}

For the  kind of problems treated here, the algebra  {\rm End}$_F(V)$ with $V$ infinite-dimensional seems to be easier to deal with 
than the matrix algebra $M_n(F)$. Indeed, all its elements are commutators and it has no polynomial identities,   so there is no need to distinguish between  
different polynomials. One may thus wonder whether $f({\rm End}_F(V))$   is actually equal to {\rm End}$_F(V)$ for any nonconstant polynomial $f$. The next example 
shows that this is not true (although $f({\rm End}_F(V))$  is indeed rather large---see \cite{CL}).

\begin{example}\label{enov}
Let $V$ be a countably infinite-dimensional vector space over a field $F$,
and let $A={\rm End}_F(V)$. Take any basis
  $\{e_1,e_2,\dots\}$  of $V$.
%Let $r\in A$ be the right shift operator, i.e., $r$ is defined by $r(e_n) = e_{n+1}$ for every $n\ge 1$. Pick any $f\in A$ and define $g\in A$ by $g(e_1)=0$ and 
%$$g(e_n) = r^{n-2}f(e_1) - r^{n-3}f(e_2) - ...-rf(e_{n-2}) - f(e_{n-1})$$
%for any $n\ge 2$. Then 
 %$f=[r,g]$. Thus, every element in $A$ is a commutator; in particular, $A=[A,A]$.
Let $\ell\in A$ be the left shift operator, i.e., $\ell$ is defined by $\ell(e_1)=0$ and $\ell(e_n) = e_{n-1}$ for every $n\ge 2$. We claim that $\ell$ is not equal to the
square  of an element in $A$. Suppose this is not true. Let $h\in A$ be such that $\ell = h^2$. Write
$$h(e_1)= \lambda_1 e_1 + \lambda_2 e_2 + \dots + \lambda_r e_r,$$
 where $\lambda_i\in F$. Since $\ell=h^2$ commutes with $h$, $\ell( h(e_1))=0$. Hence,  $\lambda_2 e_1 + \dots + \lambda_r e_{r-1}=0$, and so $\lambda_2=\dots=\lambda_r=0$. Thus,
$h(e_1)=\lambda_1 e_1$. However, since $h^2(e_1)=\ell(e_1)=0$, this is possible only when $h(e_1)=0$. We continue by examining $h(e_2)$. Let $\mu_i\in F$ be such that
$$h(e_2)= \mu_1 e_1 + \mu_2 e_2 + \dots + \mu_s e_s.$$
Since $$\ell (h(e_2))  = h(\ell (e_2)) = h(e_1)=0,$$ it follows that  $\mu_2 e_1 + \dots + \mu_s e_{s-1}=0$. Therefore, $\mu_2=\dots=\mu_s=0$ and so
$h(e_2)= \mu_1 e_1$. But then
$$e_1= \ell (e_2) = h(h(e_2)) = \mu_1 h(e_1)=0,$$
a contradiction. 

 We have thus shown that 
  $\ell$ does not lie in the image of the polynomial $f(X)=X^2$. In particular, $f({\rm End}_F(V)) \ne {\rm End}_F(V)$. We remark that this is still true if $V$ is 
	finite-dimensional. Indeed, 
	 a nilpotent matrix $a\in M_n(F)$ of maximal  nilindex cannot be written as $b^2$ for some $b\in M_n(F)$. 
To prove this, observe that
$a=b^2$ implies $b^{2n} = a^n=0$ and hence $b^n=0$, which leads to the contradiction that $a^{n-1} =b^{2n-2}=0$.
\end{example}
 
\begin{corollary}\label{cH} Let $H$ be an infinite-dimensional Hilbert space and
let  $f\in \mathbb C\langle \mathcal X\rangle$ be a  nonconstant polynomial. Then every element in
 $A=B(H)$ is 
%a difference of two elements, each of which is 
a sum of $3916$ elements from $f(A)-f(A)$. 
\end{corollary}

\begin{proof}
As in the preceding proof, pick an
 $n\ge 2$ such that $f$  is neither an identity nor a central polynomial of $M_n(F)$. As is well-known, $A\cong M_n(A)$. Further, all operators in $A$ except those of the form $\lambda I + K$ with $\lambda\in\mathbb C\setminus{\{0\}}$ and $K$ compact are commutators \cite{BP}; in particular, compact operators are commutators, so every element in $A$ is a sum of a commutator and a central element (i.e., a scalar multiple of the identity). Theorem \ref{tt2} thus tells us that every commutator in $A$ is
%a difference of two elements, each of which is 
a sum of $1958$ elements from $f(A)-f(A)$. 
 Finally, we use the fact that every operator in $A$ is a sum of two commutators \cite{Hal}.
\end{proof}

Our second main result of this section
follows immediately from  Theorem \ref{tder} and Lemma \ref{ls}\,(d).
 %Its proof is based on  Theorem \ref{tder}.

\begin{theorem}\label{secondmain}
Let $F$ be an algebraically closed  field with {\rm char}$(F)=0$, let $A=M_n(F)$ with $n\ge 2$, and let $f\in  F\langle \mathcal X\rangle$ be a   polynomial  which is neither an identity nor a central polynomial of $A$.
Then $f(A)-f(A)$ contains all square-zero matrices in $A$. 
\end{theorem}

%\begin{proof}
%Apply Theorem \ref{tder} and Lemma \ref{ls}\,(d).
%\end{proof}

%\begin{proof} According to Theorem \ref{tder}, we may choose a
 %$t\in f(A)$ such that the algebraic multiplicity of any of its eigenvalues does not exceed $\frac{n}{2}$. \end{proof}

 Example \ref{ekk} shows that $f(A)$ itself does not need to contain any nonzero nilpotent matrix. 
The involvement of at least two elements from $f(A)$ is thus necessary, so Theorem \ref{secondmain} provides the best possible result of this kind. 
However, the question of  what can be said about other matrices, i.e., those whose square is not $0$, remains open.

\begin{corollary}\label{ts2}
Let $F$ be an algebraically closed  field with {\rm char}$(F)=0$, let $A=M_n(F)$ with $n\ge 2$,
  and let $f\in F\langle \mathcal X\rangle$ be a   polynomial  which is neither an identity nor a central polynomial of $A$. Then every traceless matrix in $A$ is
% a difference of two elements, each of which is 
a sum of four matrices from $f(A)-f(A)$. 
\end{corollary}

\begin{proof}
Apply Theorem \ref{secondmain} and \cite[Theorem 1.1]{dP} which states that every traceless matrix is a sum of four square-zero matrices. 
\end{proof}

We say that polynomials $f,g\in F\langle \mathcal X\rangle$ are {\em cyclically equivalent} if $f-g$ is a sum of commutators in 
$F\langle \mathcal X\rangle$.

\begin{corollary}\label{cs2}Let $F$ be an algebraically closed  field with {\rm char}$(F)=0$, let $A=M_n(F)$ with $n\ge 2$,  and let $f\in F\langle \mathcal X\rangle$ be a   polynomial  which is 
 not cyclically equivalent to an identity of $A$
and is
not  a central polynomial of $A$. Then every matrix in $A$ is a linear combination of nine matrices from $f(A)$.
\end{corollary}

\begin{proof}
By \cite[Corollary 4.6]{BK1}, there exists an $a\in f(A)$ whose trace is not $0$. Write $x\in A$ as $$x= \frac{{\rm tr}(x)}{{\rm tr}(a)}a + \Bigl(x - \frac{{\rm tr}(x)}{{\rm tr}(a)}a\Bigr)$$
and apply Corollary \ref{ts2} to the traceless matrix $x - \frac{{\rm tr}(x)}{{\rm tr}(a)}a$.
\end{proof}

\begin{remark}
This corollary also holds  for fields that are not algebraically closed, but we have to replace  the number $9$  by $2\cdot 68 + 1=137$ (see Remark \ref{r6}).
\end{remark}

\begin{example}\label{etr} Let char$(F)=0$.
If $f=[X_1,X_2]+ \frac{1}{n}$, then every matrix
in $f(M_n(F))$ has trace equal to $1$. Accordingly, only matrices whose trace is an integer multiple of $1$ lie in the additive span of $f(M_n(F))$. This shows that, unlike in other results of this section, the involvement of linear combinations with coefficients in the field $F$ is necessary in Corollary \ref{cs2}. Moreover, it also shows that the main
 theorems  and their corollaries must  involve differences (rather than only sums) of elements from the image of a polynomial. Indeed, a sum 
of matrices from $f(M_n(F))$ is never a traceless matrix.  

It may be more illuminating to give an example of a polynomial with zero constant term:  if $g=[X_1,X_2] + [X_1,X_2]^4$, then a sum of matrices from $g(M_2(\mathbb R))$
 is never a matrix with negative trace.
\end{example}

It is well-known (and easy to see) that a central polynomial of $M_n(F)$ is an identity of $M_k(F)$ for every
$k < n$. Hence, if a polynomial   is neither an identity nor a central polynomial of $M_2(F)$, then
the same holds for $M_n(F)$ for every $n\ge 2$.
The next result therefore follows easily from Corollary \ref{ts2}, so we state it without 
proof.

\begin{corollary}\label{ts2a}Let $F$ be an algebraically closed field of characteristic $0$  and let $f\in F\langle \mathcal X\rangle$ be a   polynomial  which is neither an identity nor a central polynomial of $M_2(F)$.
If  $A =\prod_{n=2}^\infty M_n(F)$ (i.e., $A$ is the direct product of all matrix algebras $M_n(F)$ with $n\ge 2$), then every commutator in $A$
is
 a sum of four elements from $f(A)-f(A)$.
In particular,
$[A,A]\subseteq  {\rm span}\,f(A)$.
\end{corollary}

It is rather obvious that various other versions of this corollary follow from Theorem \ref{tt2} and its corollaries. We wanted to give just a sample  result indicating the applicability of the fact that the number of summands from the image of a polynomial, needed to represent any commutator,  is bounded.
Note that,
in general, there is no reason to believe that the condition that a family of algebras 
 $(A_n)_{n\ge 1}$  and  a polynomial $f$ satisfy 
$[A_n,A_n]\subseteq  {\rm span}\,f(A_n)$ for every $n$ implies that the direct product
$A=\prod_{n=1}^\infty A_n$ satisfies $[A,A]\subseteq  {\rm span}\,f(A)$. For example, if $f$ is multilinear and so the linear span of the image of $f$  coincides with its additive span, and $A_n$ is such that $[A_n,A_n]\subseteq  {\rm span}\,f(A_n)$ (resp.  $A_n= {\rm span}\,f(A_n)$)
but some commutators (resp. elements) in $A_n$ cannot be expressed as sums of less than $n$ elements from
$f(A_n)$, then $[A,A]\not\subseteq  {\rm span}\,f(A)$ (resp.  $A\ne {\rm span}\,f(A)$). For a concrete example of such algebras $A_n$ and polynomials $f$, see \cite[Example 3.11]{R}. We remark that the existence of such examples also shows that the results of Section \ref{s4} do not hold for  all algebras treated in Section \ref{s3}.

%the this is true if and only if there exists an $N\ge 1$ such that every element in $[A_n,A_n]$ is a sum of at most $N$ elements from $f(A_n)$.

\subsection{An analytic supplement} Suppose a polynomial $f$ is neither an identity nor a central polynomial of $A=M_n(\mathbb C)$. Does $f(A)$ contain a matrix all of whose eigenvalues are distinct? 
%Although this is not true in general (see Remark \ref{r2c}),  it is safe to say  that most polynomials   have this property.
If $n$ is a prime number, then the answer is affirmative by \cite[Proposition 2.7]{KRZ} (for 
$n=2$ and $n=3$, this follows also from Theorem \ref{tder}). % shows that the answer is affirmative.
 It is also affirmative if $f$ is multilinear
\cite[Theorem 1.8]{Retc}.  Although it is %We do not know whether  it is 
not affirmative in general (see Remark \ref{r2c}), it is safe to say  that most polynomials   have this property.

By $\overline{\mathbb C\big(f(A)-f(A)\big)}$ we denote the closure (with respect to the usual matrix metric) of the set of all scalar multiples of matrices from
$f(A)-f(A)$.

\begin{theorem}
Let $A=M_n(\mathbb C)$ with $n\ge 2$. If $f\in  \mathbb C\langle \mathcal X\rangle$ is such that $f(A)$ contains a matrix all of whose eigenvalues are distinct, %(e.g., if $n$ is a prime  or $f$ is multilinear),
then  $ \overline{\mathbb C\big(f(A)-f(A)\big)}$ contains all traceless matrices.
\end{theorem}

\begin{proof}
Since $f(A)$ is invariant under conjugation by invertible matrices, 
it  contains a diagonal matrix $d$  having distinct diagonal entries. Therefore, $f(A)-f(A)$ 
contains every matrix of the form
\begin{align*}
&e^{-\lambda x} d e^{\lambda x} - d\\ =& \Big(1 -  \lambda x + \frac{\lambda^2}{2!} x^2 -\dots\Big)d \Big(1 +  \lambda x + \frac{\lambda^2}{2!} x^2 +\dots\Big)-d\\
=& \lambda [d,x] + \frac{\lambda^2}{2!} [[d,x],x] + \dots,
\end{align*}
where $\lambda\in \mathbb C$ and $x\in A$. Hence,
$$[d,x] = \lim_{\lambda\to 0} \frac{1}{\lambda}\big( e^{-\lambda x} d e^{\lambda x} - d\big) \in \overline{\mathbb C\big(f(A)-f(A)\big)}.$$
An easy matrix calculation shows that every matrix with zeros on the diagonal can be written as $[d,x]$ for some $x\in A$. Since every traceless matrix is (even unitarily) similar
to such a matrix 
\cite[Problem 3 on p.\,77]{HJ}, the desired conclusion follows from the fact that $\overline{\mathbb C\big(f(A)-f(A)\big)}$ is invariant under conjugation.
\end{proof}

We do not know whether  taking the closure of $\mathbb C\big(f(A)-f(A)\big)$ is necessary.

\subsection{Concluding remarks} 
The above results basically show that for any  of the algebras
 $A=M_n(F)$,  $A=B(H)$, etc.,
  there exists a positive integer $N$ such that every  commutator (or even every element)  in $A$
%of one of the algebras $A=M_n(F)$,  $A=B(H)$, etc.,
 can be expressed by $N$ elements from the image of any polynomial $f$ that satisfies  the necessary restrictions. The main point is that $N$ is independent of $f$ %(which may be of any degree) 
as well as of the size of the matrices. However, {\em what is  the  minimal $N$}? 
More precisely, one can ask the following question.

\begin{question}\label{q25}
 %In each of Corollaries \ref{ts22},  \ref{cend},  \ref{cH}, \ref{ts2}, and \ref{cs2},  
What is the smallest number  that can replace  the %multi-digit  
number appearing in the statement of Corollary \ref{ts22},  \ref{cend},  \ref{cH}, \ref{ts2}, and \ref{cs2},  respectively  (i.e.,  $7788$, $1958$, $3916$, $4$, and $9$, respectively)?
\end{question}

This paper was essentially devoted  to the proof of the existence of this  number, but its determination is left as an open problem. The above four-digit numbers were obtained by a  somewhat rough method,
so we conjecture that they can be substantially lowered. The method that led to the numbers $4$ and $9$ was more sophisticated, but nevertheless it is hard to believe that they are the smallest 
possible.

The L'vov-Kaplansky conjecture states that, in the settings of Corollaries \ref{ts2} and \ref{cs2}, $N=1$ for multilinear polynomials (which certainly is not true for general polynomials, see
 Examples \ref{ekk},  \ref{enov}, and \ref{etr}).
  One can of course also ask what is the minimal $N$ for images of  multilinear polynomials in algebras from Corollaries \ref{ts22},  \ref{cend},  and \ref{cH}.
%As pointed out several times (in particular, in Examples \ref{ekk},  \ref{enov}, and \ref{etr}), this
 %certainly does not hold for general polynomials.
%Still, it is plausible that the minimal
 %$N$ is a rather small number even for general polynomials---certainly much smaller than 
%the big numbers  obtained by our somewhat rough method. 
%This paper, however, was devoted only to the proof of the existence of this  number, while its determination is left as an open problem.

It seems that  Question \ref{q25} can be quite difficult 
even if we fix the polynomial. For example, a still active area of research is  Waring's problem for matrices, which asks about  the number of summands needed to express a given matrix in $M_n(C)$, where 
$C$ is a commutative ring, as a sum of $k$th powers of some matrices in $M_n(C)$ (see
\cite{KG, L} and references therein). This is, of course, an extension of the classical problem, proposed by Waring in $18$th century and solved affirmatively by Hilbert,
asking whether for each positive integer $k$ there exists a positive integer $g(k)$  such that
 every positive integer is a sum of at most $g(k)$ $k$th powers.
%(for $n=1$ and $ R =\mathbb N$, this is the classical Waring's problem from the $18$th century, which was solved by Hilbert). 
  It may be remarked that  Waring type problems  on $k$th powers have   been studied  in various  rings and algebras, not only
	in $M_n(C)$ (see, e.g., \cite{LW, V}). 

Finally, 
we mention the analogy of the proposed problem with the Waring problem on images of words in groups, which was, in particular, solved for
finite simple groups 
\cite{LST}. Specifically, it was shown that given a word  $w=w(x_1,\dots,x_d) \ne 1$
and denoting, for any group $\Gamma$, by 
$w(\Gamma)$ the image of the word map  $\Gamma_w:\Gamma^d\to \Gamma$, we have $w(\Gamma)^2 = \Gamma$  for 
every  finite non-abelian simple group $\Gamma$ of sufficiently high order.   
One can thus say that the solution of the  analogous group-theoretic problem is $N=2$.\\

%\smallskip

\noindent
{\bf Acknowledgments.} The author is thankful to Peter \v Semrl for his help in some linear algebra considerations, to Jurij Vol\v ci\v c for drawing his attention to $2$-central polynomials and the paper \cite{KRZ}, to Igor Klep for useful general comments, and to Alexei Kanel-Belov for pointing out his paper \cite{KB}  after seeing the first
arXiv version of this paper.


\begin{thebibliography}{99}

\bibitem{ABESV}  J.\,Alaminos, M.\,Bre\v sar, J.\,Extremera, \v S.\,\v Spenko, A.\,R.\,Villena,
Commutators and square-zero elements in Banach algebras, {\em Quart. J. Math.} {\bf 67} (2016), 1--13.

\bibitem{AM}A.\,Albert, B. Muckenhoupt, On matrices of trace 0, {\em Michigan Math. J.} {\bf 1} (1957), 1--3.


%\bibitem{A}S.\,A.\,Amitsur,
%Nil PI-Rings,  {\em Proc. Amer. Math. Soc.} {\bf  4} (1951),
%538--540.


\bibitem{A2}S.\,A.\,Amitsur,
The T-ideals of the free ring,
 {\em J. London Math. Soc.} {\bf 30} (1955), 470--475.

\bibitem{INCA}M.\,Bre\v{s}ar, {\em Introduction to Noncommutative Algebra}, Universitext, Springer, 2014.





\bibitem{BKS}
M.\,Bre\v sar, E.\,Kissin, V.\,Shulman, Lie ideals: from pure algebra to $C^*$-algebras, {\em J. Reine Angew. Math.} {\bf 623} (2008), 73--121.

\bibitem{BK1}
M.\,Bre\v sar, I.\,Klep, Values of noncommutative polynomials, Lie skew-ideals and tracial Nullstellens\" atze, {\em Math. Res. Lett.} {\bf 16} (2009), 605--626.

\bibitem{BK2}
M.\,Bre\v sar, I.\,Klep,
A note on values of noncommutative polynomials, {\em Proc. Amer. Math. Soc.} {\bf  138} (2010),  2375--2379.

\bibitem{BK3}
M.\,Bre\v sar, I.\,Klep,
A local-global principle for linear dependence of noncommutative polynomials, {\em Israel J. Math.} {\bf 193} (2013), 71--82.

\bibitem{BPS}  M.\,Bre\v sar, C.\,Procesi, \v S.\,\v Spenko, 
Quasi-identities on matrices and the Cayley-Hamilton polynomial, {\em Adv. Math.} {\bf  280} (2015), 439--471.

\bibitem{BP}
A.\,Brown, C.\,Pearcy, Structure of commutators of operators, {\em Ann. Math.} {\bf 82}
(1965), 112--127.

\bibitem{C}
C.-L.\,Chuang, 
On ranges of polynomials in finite matrix rings, {\em
Proc. Amer. Math. Soc.} {\bf 110} (1990),  293--302.

\bibitem{CL}
C.-L.\,Chuang, T.-K.\,Lee,
Density of polynomial maps, {\em Canad. Math. Bull.} {\bf  53} (2010), 223--229.



\bibitem{Hal} P.\,R.\,Halmos,  Commutators of operators II, {\em Amer. J. Math.} {\bf 76} (1954), 191--198.


\bibitem{Her2} I.\,N.\,Herstein, 
Sui commutatori degli anelli semplici,  {\em Rend. Sem. Mat. Fis. Milano} {\bf 33} (1963), 
80--86.
\bibitem{Her} I.\,N.\,Herstein, {\em Topics in Ring Theory}, The University of Chicago Press, Chicago, 1969.

\bibitem{HJ}
R.\,A.\,Horn, C.\,Johnson, {\em Matrix Analysis}, Cambridge Univ. Press, 1985.


\bibitem{KB}A.\,Kanel-Belov,
No associative PI-algebra coincides with its commutant, {\em Siberian Math. J.}  {\bf 44} (2003), 969–980.

\bibitem{KBMR}A.\,Kanel-Belov,  S.\,Malev,  L.\,Rowen, 
The images of non-commutative polynomials evaluated on $2\times 2$ matrices,
{\em Proc. Amer. Math. Soc.} {\bf 140} (2012),  465--478.



\bibitem{KRZ}A.\,Kanel-Belov, F.\, Razavinia, W.\,Zhang,
Centralizers in free associative algebras
and generic matrices, preprint, arXiv: 11812.03307.



\bibitem{KG}S.\,A.\,Katre, A.\,S.\,Garge,  Matrices over commutative rings as sums of $k$-th powers, {\em Proc. Amer. Math. Soc.} {\bf 141} (2013),  103--113. 

\bibitem{LM}
C.\, Lanski, S.\,Montgomery, 
Lie structure of prime rings of characteristic $2$,
{\em Pacific J. Math.} {\bf 42} (1972), 117--136.

\bibitem{LST}M.\,Larsen, A.\,Shalev, P.\,H.\,Tiep, 
The Waring problem for finite simple groups,
{\em Ann.  Math.} {\bf  174} (2011),  1885--1950.


\bibitem{L}J.\,Lee, Integral matrices as diagonal quadratic forms,
{\em Linear Multilinear Algebra} {\bf 66} (2018),  742--747.

\bibitem{LZ}T.-K.\,Lee, Y.\,Zhou, Right ideals generated by an idempotent of finite rank, {\em Linear Algebra Appl.} {\bf 431} (2009), 2118--2126.

\bibitem{LW}Y.-R.\,Liu, T.\,D.\,Wooley,  Waring's problem 
in function fields, {\em J. Reine Angew. Math.} {\bf 638} (2010), 1--67. 

\bibitem{MM}
W.\,S.\,Martindale 3rd,  C.\,R.\,Miers, On the iterates of derivations of prime rings, {\em  Pacific J. Math.} {\bf 104} (1983), 179--190.

\bibitem{Mat}K. Matsumoto, 
Selfadjoint operators as a real span of $5$ projections,
{\em Math. Japon.} {\bf 29} (1984),  291--294.



\bibitem{M} Z.\,Mesyan, 
Commutator rings,  {\em Bull. Austral. Math. Soc.} {\bf 74} (2006),  279--288.

\bibitem{PT}C. Pearcy, D. Topping, 
Sums of small numbers of idempotents,
{\em Michigan Math. J.} {\bf 14} (1967), 453--465.


\bibitem{Pop} 
C.\,Pop, Finite sums of commutators, {\em Proc. Amer. Math. Soc.} {\bf 130} (2002), 3039--3041.

%\bibitem{Pos} E.\,C.\,Posner,
%Prime rings satisfying a polynomial identity,
%{\em Proc. Amer. Math. Soc.} {\bf 11} (1960), 180--183.

\bibitem{R}L.\,Robert, On the Lie ideals of $C^*$-algebras, {\em J. Operator Theory} {\bf 75} (2016), 387--408.

\bibitem{Retc}  L.\,Rowen, R.\,Yavich,  S.\,Malev,  A.\,Kanel-Belov,  Evaluations of noncommutative polynomials on finite dimensional algebras. 
L'vov-Kaplansky conjecture, preprint, arXiv: 1909.07785.


\bibitem{dP}
C.\,de\,Seguins\,Pazzis, A note on sums of three square-zero matrices, {\em Linear Multilinear Algebra} {\bf  65} (2017),  787--805. 


\bibitem{S}A. Stasinski,
Similarity and commutators of matrices over principal ideal rings, {\em
Trans. Amer. Math. Soc.} {\bf 368} (2016), 2333--2354.

\bibitem{Sp}\v S.\,\v Spenko,  On the image of a noncommutative polynomial, {\em J. Algebra} {\bf 377} (2013), 298--311.


\bibitem{V}L.\,N.\,Vaserstein,  Waring's problem for algebras over fields, {\em J. Number Theory} {\bf 26} (1987), 286--298.
%https://arxiv.org/pdf/1909.07785.pdf

\bibitem{W}G.\,Weiss,  $B(H)$-commutators: a historical survey, {\em Recent advances in operator theory, operator algebras, and their applications}, 307--320, Oper. Theory Adv. Appl. {\bf 153}, Birkh\" auser,  2005. 

\bibitem{Wu}P.\,Y.\,
Wu, 
Additive combinations of special operators, {\em Functional analysis and operator theory},  337–361,
Banach Center Publ. {\bf 30}, Polish Acad. Sci. Inst. Math., Warsaw, 1994.

 \end{thebibliography}
\end{document}